\documentclass{proc-l}
 \pdfoutput=1

\usepackage{amssymb, amsmath, mathabx}

\usepackage{graphicx}
\usepackage{enumerate}

\usepackage{csquotes, subcaption, xspace}
\usepackage[capitalise,noabbrev]{cleveref}

\newtheorem{theorem}{Theorem}[section]
\newtheorem{lem}[theorem]{Lemma}

\theoremstyle{definition}
\newtheorem{defn}[theorem]{Definition}

\newtheorem{prop}[theorem]{Proposition}
\newtheorem{cor}[theorem]{Corollary}

\theoremstyle{remark}
\newtheorem{remark}[theorem]{Remark}

\numberwithin{equation}{section}
\newcommand{\ds}{\displaystyle}

\newcommand{\Trap}{\mathrm{Trap }}

\newcommand{\Int}[0]{\mathrm{Int}\hspace{2pt}}

\newcommand{\K}{\mathcal{K}}
\newcommand{\F}{\mathcal{F}}
\newcommand{\T}{\mathcal{T}}
\newcommand{\M}[0]{\mathcal{M}}

\newcommand{\Hid}{\mathcal{H}}
\newcommand{\Reals}{\mathbb{R}}

\newcommand{\Tvectors}{T}
\newcommand{\Fvectors}{F}

\newcommand{\that}{\widehat{t}}
\newcommand{\tch}{\widecheck{t}}
\newcommand{\h}{\textbf{h}}
\newcommand{\p}{\textbf{p}}

\newcommand{\minus}{\scriptscriptstyle{-}}
\newcommand{\plus}{\scriptscriptstyle{+}}

\newcommand{\focusA}{A}
\newcommand{\focusB}{B}

\newcommand{\bflow}{\phi}

\newcommand{\paramise}{parameterise\xspace}
\newcommand{\paramised}{parameterised\xspace}
\newcommand{\paramation}{parameterisation\xspace}

\begin{document}

\title[Unilluminable rooms]{Unilluminable rooms, billiards with hidden sets, and Bunimovich mushrooms}


\author[Paul Castle]{Paul Castle}
\address{Department of Mathematics, University of Western Australia, Perth WA 6907}
\email{paul@madgech.com}
\thanks{The author has previously published under the name Paul Wright}



\date{\today}


\commby{Guofang Wei}

\begin{abstract}
	The illumination problem is a popular topic in recreational mathematics: In a mirrored room, is every region illuminable from every point in the region? So-called \enquote{unilluminable rooms} are related to \enquote{trapped sets} in inverse scattering, and to billiards with divided phase space in dynamical systems. In each case, a billiard with a semi-ellipse has always been put forward as the standard counterexample: namely the Penrose room, the Livshits billiard, and the Bunimovich mushroom respectively. In this paper, we construct a large class of planar billiard obstacles, not necessarily featuring ellipses, that have dark regions, hidden sets, or a divided phase space. The main result is that for any convex set $\Hid$, we can construct a convex, everywhere differentiable billiard table $K$ (at any distance from $\Hid$) such that trajectories leaving $\Hid$ always return to $\Hid$ after one reflection. This billiard generalises the Bunimovich mushroom. As corollaries, we give more general answers to the illumination problem and the trapped set problem. We use recent results from nonsmooth analysis and convex function theory, to ensure that the result applies to all convex sets. 
\end{abstract}

\maketitle

\section{Introduction}

In this paper we consider three closely related problems in optics and dynamical billiards:

\begin{enumerate}
	\item \label{illumination problem} \emph{The illumination problem}: In a mirrored room (or closed billiard), is every region illuminable from a candle placed at any point in the room?
	\item \label{trapped set problem} \emph{The trapped set problem}: Does the scattering kernel of an open billiard determine the shape of a billiard obstacle?
	\item \label{divided phase space} \emph{Divided phase space}: Which closed billiards have a phase space divided into isolated components?
\end{enumerate}
All three problems have similar answers involving a semi-ellipse. They use the property that any billiard trajectory between the two focii will be reflected by the ellipse back through the focii. 

\subsection{Illumination problem}
The first question is thought to have been first asked in the 1950s by Straus \cite{klee1969every,chernov2003search}, and answered in the negative by Penrose \cite{penrose1958puzzles}. 

Penrose's solution uses a semi-elliptical room similar to \cref{fig:Livshits}(\textsc{a}) and (\textsc{b}). Variations of Penrose's solution with chains of ellipse-based rooms have been considered \cite{rauch1978illumination}. One variation of the question replaces the candle with a searchlight \cite{chernov2003search}. Other than ellipse-based answers, most work in this area has been on polygonal rooms. There are several polygonal counterexamples \cite{tokarsky1995polygonal, castro1997corrections}, which have two points that cannot illuminate each other. In a rational polygon, only finitely many points can remain dark \cite{lelievre2016everything}. The problem has been included on various lists of unsolved problems \cite{guy1971monthly,klee1979some,klee1991old}, and featured on popular recreational mathematics websites \cite{numberphile}.

\subsection{The trapped set problem}
The second question was answered in the negative by Livshits, whose counterexample \cref{fig:Livshits}(\textsc{a}) was published by Melrose in \cite{melrose1995geometric}. Inverse scattering is the problem of recovering the shape of an obstacle from its scattering kernel or scattering length spectrum \cite{stoyanov2000scattering, noakes2015travelling, noakes2015rigidity}.
The Livshits example demonstrates that there exist simply connected billiard obstacles with the property that some set of points is \emph{hidden} from the outside; that is, all trajectories through these points are \emph{trapped} and will never escape. This means that inverse scattering is impossible in this case: billiard trajectories cannot provide any information about the shape of the obstacle where it borders the hidden set. It is therefore interesting to know whether billiards with hidden sets are \enquote{common}, or if Livshits-like billiards are a special case. 

By rotating the semi-ellipse around an axis, one can construct billiards with trapped sets in any dimension \cite{noakes2016obstacles}. Stoyanov \cite{stoyanov2017santalo} showed that sufficiently small perturbations to a billiard obstacle only change the Liouville measure of the set of trapped trajectories Trap$(\hat{\Omega})$ by a small amount. However, this theorem says nothing about the set of hidden points. It is possible that a small perturbation to the Livshits billiard could remove a set of very small measure from the trapped set, while completely destroying the hidden set. 

\begin{figure}
	\centering
	\subcaptionbox{The original Livshits billiard.
		\label{fig:OriginalLivshits}}
	{\includegraphics[width=0.4\linewidth]{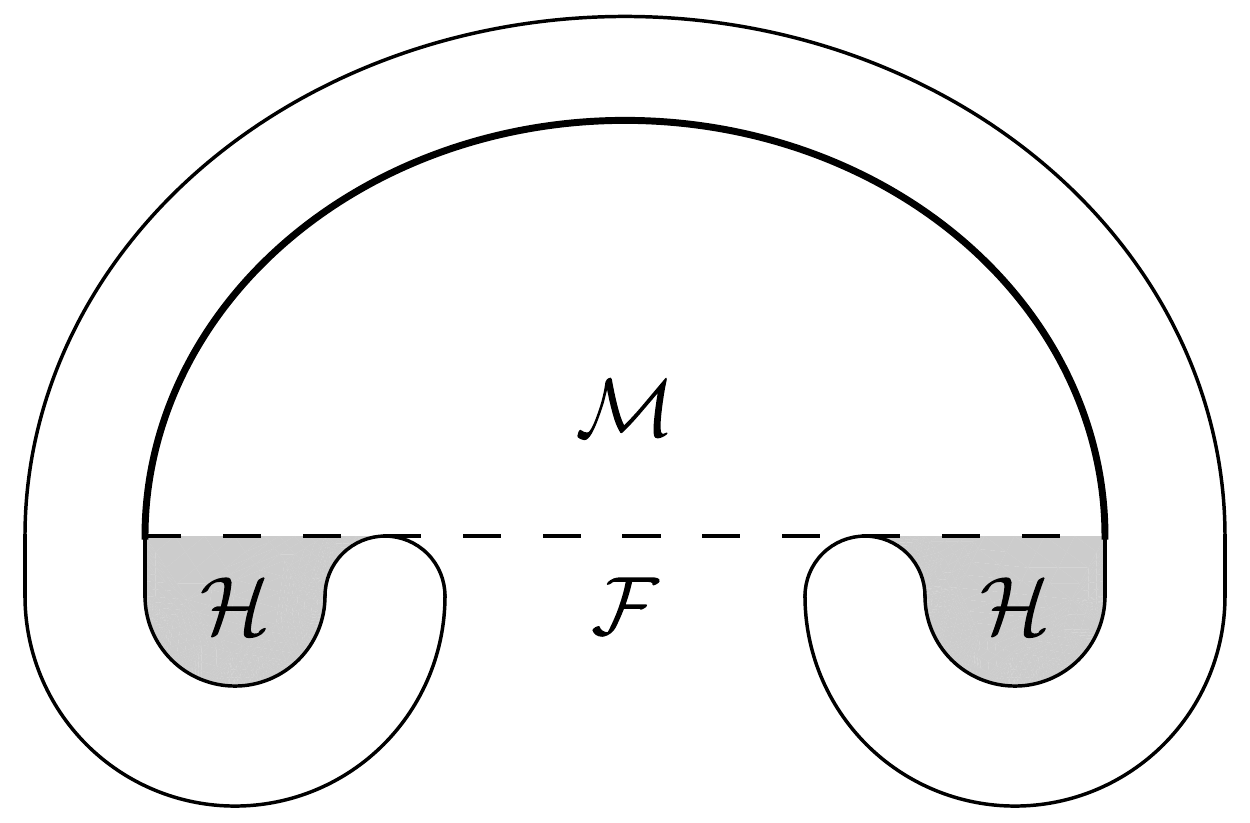}}
	\hspace{1cm}
	\subcaptionbox{\label{fig:NewLivshits}An alternative Livshits billiard.}
	{\includegraphics[width=0.4\linewidth]{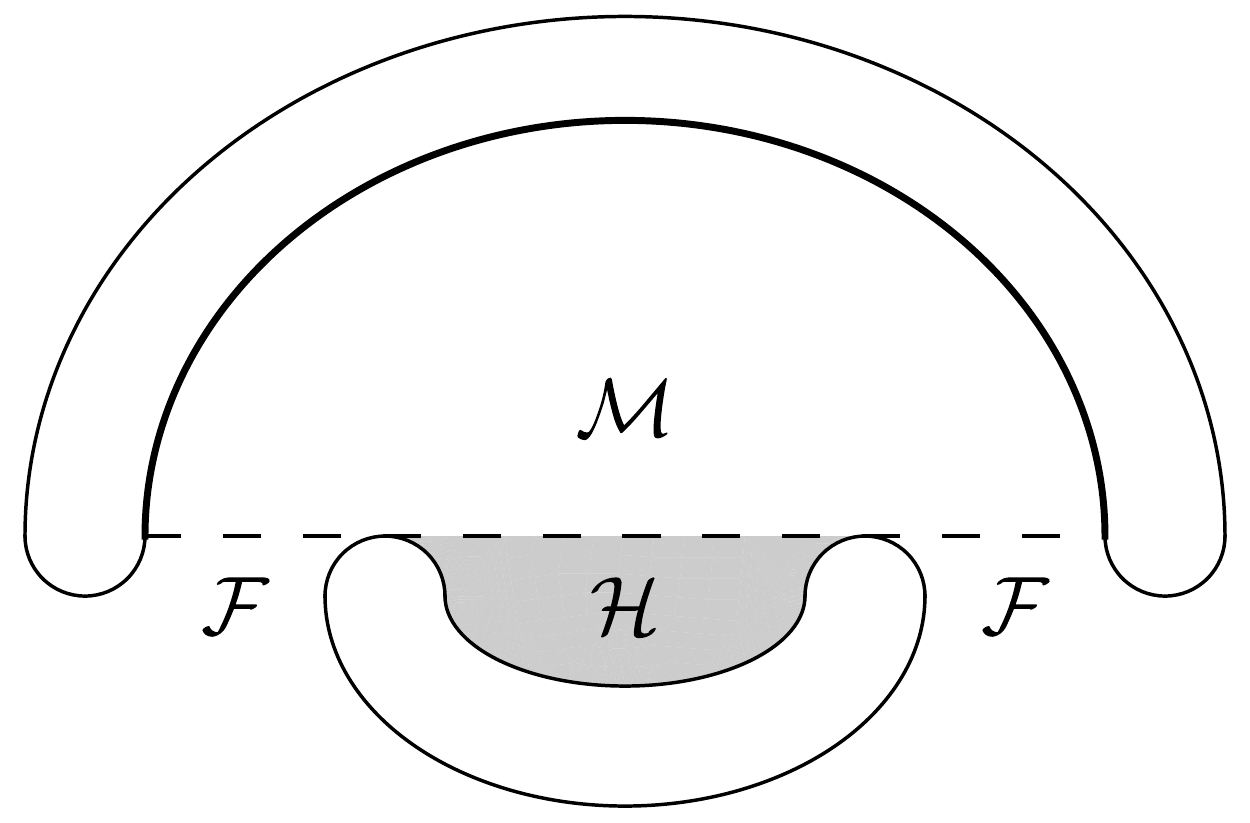}}
	\caption{Livshits billiards. The thick curve is a semi-ellipse. The rest of the curve is tangent to the major semiaxis at the focii.}\label{fig:Livshits}
\end{figure}

\subsection{Billiards with divided phase space}
Bunimovich \cite{bunimovich2001mushrooms} uses a similar semi-ellipse to construct closed billiards with multiple chaotic components and integrable islands, and calls these billiards \enquote{mushrooms}. These have been investigated in the field of quantum chaos \cite{barnett2007quantum, dietz2007spectral}. Bunimovich writes \enquote{Observe that we allowed here only semicircular and semielliptic hats\ldots perturbations of (semi) ellipses can be expected to provide a generic picture of Hamiltonian systems with divided phase space}. 

\section{Results}

Other than polygonal rooms, all of the above examples incorporate a semi-ellipse. A natural question is whether the semi-ellipse is essential for creating an unilluminable room, a hidden set, or a divided phase space in a smooth billiard. Another natural question is whether the boundary between light and dark regions is always a line segment. In this paper, we answer these questions by constructing a large, general class of planar obstacles with divided phase space, hidden sets, or unilluminable regions. 

First, for any convex set $\Hid$, we can construct a billiard around it that divides the phase space into disjoint components, one containing $\Hid$. Unlike the Penrose, Livshits and Bunimovich examples, these billiards do not necessarily use ellipses. Note that we make no assumptions about the smoothness of the set $\Hid$, beyond what is implied by the convexity. The main result is the following theorem:

\begin{theorem}\label{main theorem}
	Let $\Hid$ be a convex subset of $\Reals^2$ and let $\lambda > 0$. Then there exists a closed billiard $\K_\lambda$ surrounding $\Hid$ with the following properties:
	\begin{enumerate}
		\item $\|p -  q\| \geq \lambda$ for all $p \in K$, $q \in \Hid$.
		\item The boundary $\partial \K_\lambda$ is strictly convex, differentiable everywhere and twice differentiable almost everywhere.
		\item The phase space $\hat{\Omega}$ of the billiard flow inside $\K_\lambda$ is split into two disjoint subsets $\hat{\Omega} = \hat{\Omega}_1 \cup \hat{\Omega}_2$. Every trajectory in $\hat{\Omega}_1$ intersects $\Hid$ after every reflection, while every trajectory in $\hat{\Omega}_2$ never intersects $\Hid$. 
	\end{enumerate}
\end{theorem}

\begin{proof}[Sketch of proof]
To visualise the construction of the billiard, we can use a variation of an idea called the \enquote{goat and silo problem} \cite{fraser1982tale}. Consider a goat wearing a harness, through which a rope can move back and forth freely. We use a rope of length $L + 2 \lambda$, where $L$ is the perimeter of the silo and $\lambda > 0$. The rope is then wrapped around a silo in the shape of the set $\Hid$, but not fixed at any point, so that the goat can walk around the silo, as in \cref{fig2}(\textsc{a}). The region that the goat can reach is then is exactly the billiard table $K_\lambda$ that satisfies the conditions in \cref{main theorem}. 
\end{proof}
The curve $\partial K_\lambda$ can be thought of as a generalization of the \emph{involute} of the curve $\partial \Hid$. Involutes have found applications in optics \cite{chaves2015introduction} and mechanics. This construction generalizes the Bunimovich mushroom in that it divides the phase space, although we would not expect the dynamics in these billiards to be integrable in general. The Bunimovich mushroom itself does not follow this construction. The main theorem has two corollaries that give stronger answers to the illumination and trapped set problems. 

\begin{figure}
	\centering
	\subcaptionbox{Example for \cref{main theorem}. The goat is attached to a rope wrapped around the convex set $\Hid$, and is confined to the region $K_\lambda$. \label{fig:Goat}}
	{\includegraphics[width = 0.45\textwidth]{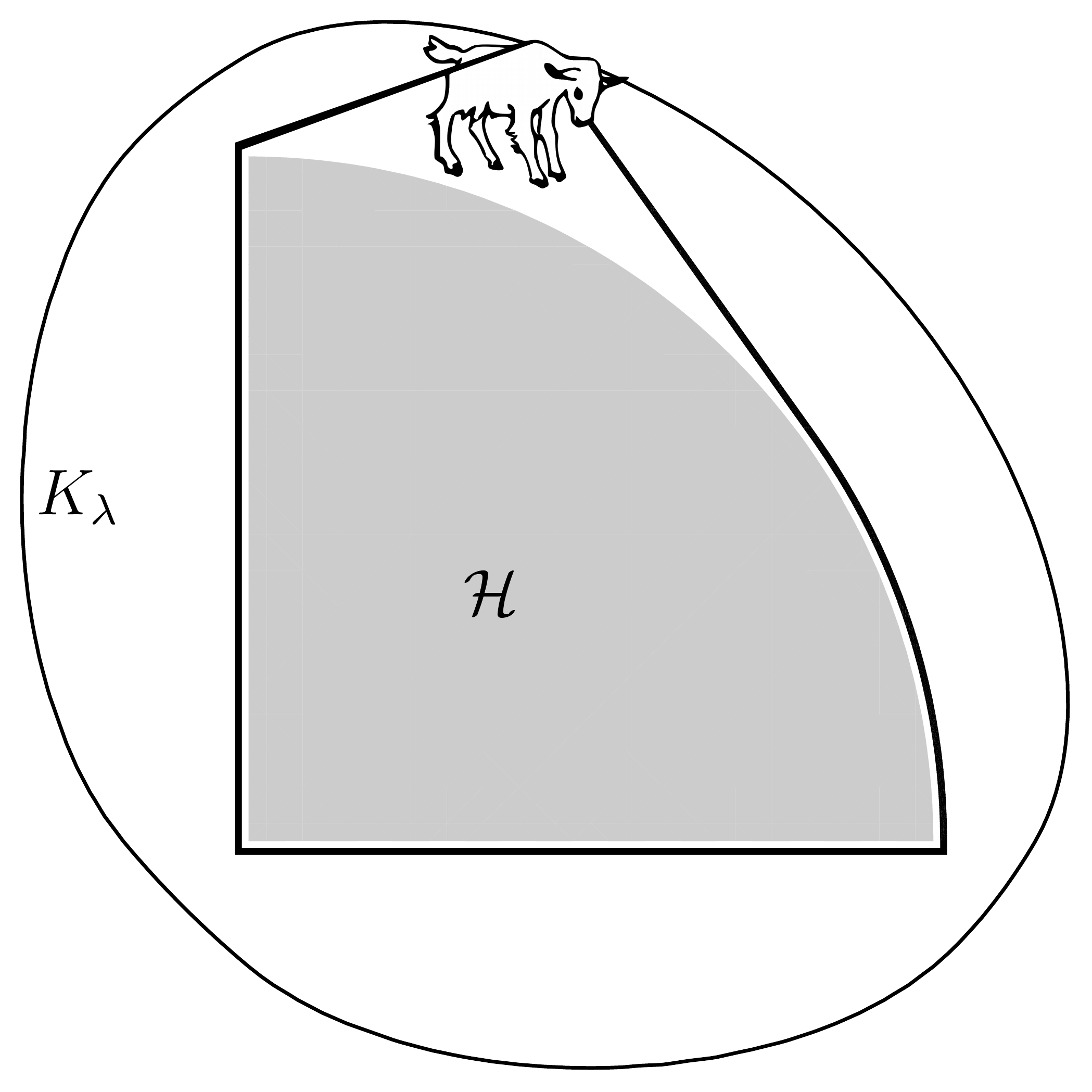}}
	\hspace{0.5cm}
	\subcaptionbox{Example for \cref{cor1}. The thick curve is determined by $\Hid$, while the rest of the boundary $\partial K_\lambda$ is arbitrary. \label{fig:CorExample}}
	{\includegraphics[width = 0.45\textwidth]{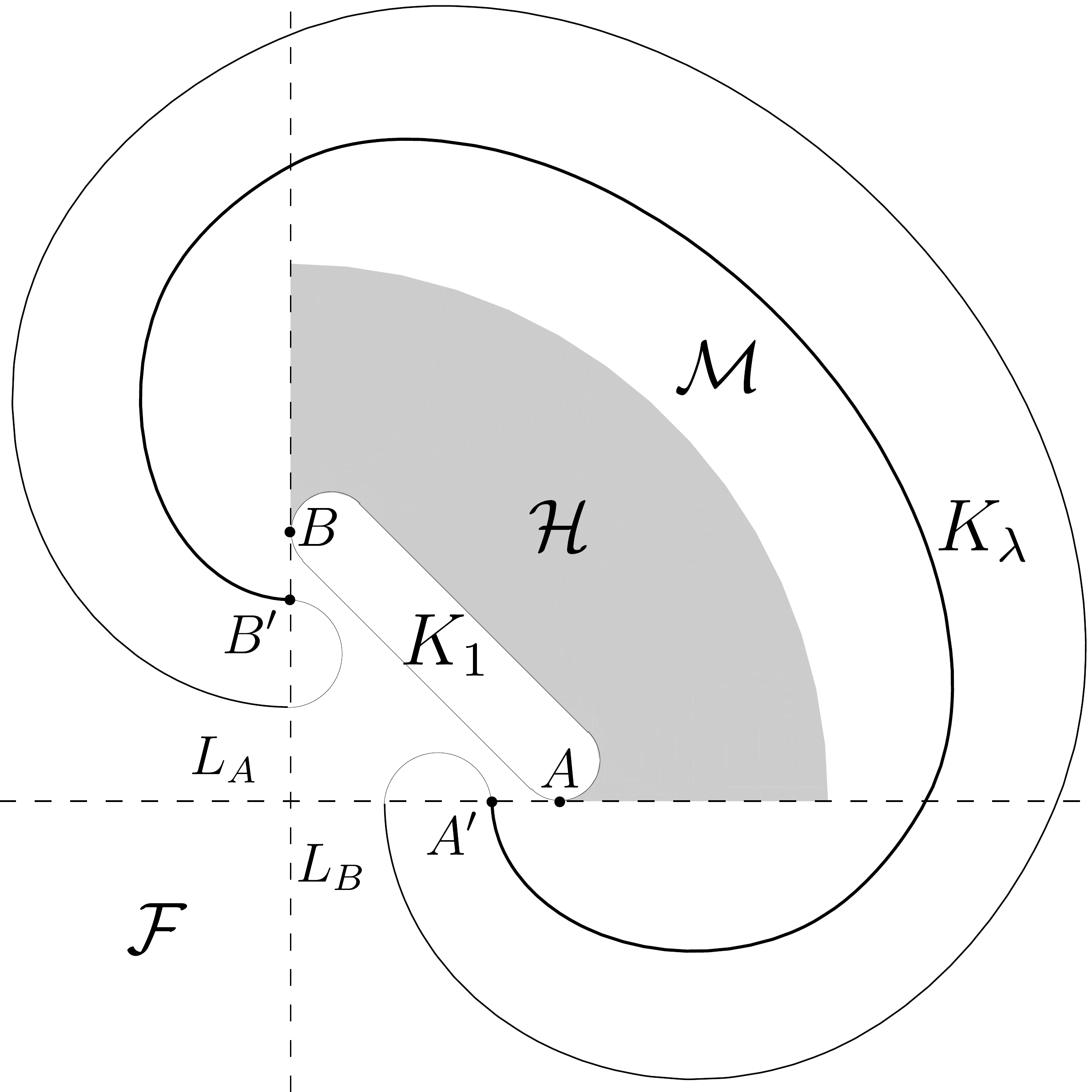}}
	\caption{\label{fig2}}
\end{figure}


\begin{cor}[Answer to the trapped set problem]\label{cor1}
	Let $\Hid$ be a convex subset of $\Reals^2$. Let $\focusA, \focusB$ be any two points on the boundary with tangent lines $L_A, L_B$. Let $K_1 \subset \Hid$ share the boundary with $\Hid$ between $\focusA$ and $\focusB$ (otherwise this component is arbitrary). Let $\lambda_{\max} = d(\Hid, L_\focusA \cap L_\focusB)$ if the lines intersect on the opposite side of $K_1$ from $\Hid$, and $\infty$ otherwise. Then for any $\lambda \in (0, \lambda_{\max})$, there exists a billiard obstacle $K_\lambda$, such that $d(\Hid, K_\lambda) \geq \lambda$, and $\Hid \backslash K_1$ is the hidden set for $K = K_1 \cup K_\lambda$. 
\end{cor}

\begin{proof}[Sketch of proof]
Following the goat and silo analogy, tie one end of the rope to point $A$ and the other end to point $B$, so that the rope passes around the silo on the opposite side to $K_1$, and again let the goat move freely along the rope. The boundary of the region the goat can access is the essential part of the billiard $K_\lambda$; the rest is arbitrary. At the points where the curve crosses the lines $L_A, L_B$, extend it back around to create the obstacle $K_\lambda$, as in \cref{fig2}(\textsc{b}). The method is similar to the so-called \enquote{gardener's ellipse} method of drawing an ellipse.
\end{proof}



\begin{cor}[Answer to the illumination problem]\label{cor2}
	Let $\Hid_1$, $\Hid_2$ be disjoint convex subsets of $\Reals^2$. Then there exists a closed billiard table $K(\Hid_1, \Hid_2)$, with arbitrarily small overlap with $\Hid_1$ and $\Hid_2$, such that, for any candle placed in $K(\Hid_1, \Hid_2)$, at least one of $\Hid_1 \backslash K$ and $\Hid_2 \backslash K$ is dark. 
\end{cor}

\begin{proof}
	Using the proof of \cref{cor1}, construct billiards $K_\lambda(\Hid_1)$ and $K_\lambda(\Hid_2)$ with sufficiently small $\lambda$, and join their openings together to form a closed billiard.
\end{proof}

\section{Preliminaries}

\subsection{Billiards}
Although in this paper we only consider billiards in $\Reals^2$, we will define them more generally here in order to introduce the concepts of hidden, free and mixed points in general. Let $g_t$ be the geodesic flow on the unit sphere bundle $SM$ of a $\mathcal{C}^k$ $(k \geq 2)$ Riemann manifold $M$ with dimension $n \geq 2$. Let $K$ be a compact subset of $M$ with $\mathcal{C}^l$  $(l \geq 2)$ boundary $\partial K$, and non-empty interior $K \backslash \partial K$. For an \emph{open billiard}, consider the compact subset $\Omega = \overline{M \backslash K}$ of $M$. For a \emph{closed billiard}, instead let $\Omega = K$. 

In either case, assume that $\Omega$ is connected. The \emph{billiard flow} $\bflow_t$ in $\Omega$ coincides with the geodesic flow $g_t$ in the interior, and when the geodesic hits the boundary $\partial K$ at $q$ with direction $v^-$, it reflects according to the law of reflection in optics:
$$v^+ = v^- - 2 \left\langle v^-, \nu(q) \right\rangle \nu(q),$$ 
where $\nu(q) \in SM$ is the unit normal vector to $\partial K$ pointing into the interior of $\Omega$. Let
$$S_q = \left \{ \begin{array}{l l}
\mathbb{S}^{n-1}	& \text{ if } q \in \Int \Omega\\
\left\{v \in \mathbb{S}^{n-1}: \left\langle v, \nu(q) \right\rangle >0 \right\}
& \text{ if } q \in \partial \Omega.
\end{array}\right.$$
Define the \emph{phase space}
$$\hat{\Omega} = \{(q, v): q \in \Int \Omega, v \in S_q\}.$$
It is well-known (e.g. \cite{cornfield1982ergodic} Section 2.4) that the geodesic flow $g_t$ preserves the Liouville measure on $SM$. The billiard flow preserves the restriction of the Liouville measure to $S\Omega$. If $\bflow_t (q, v)$ ever reaches a point on $\partial K$ that is not differentiable, or if the trajectory is ever tangent to $\partial K$, we say the trajectory is a \emph{singularity}. 

\subsection{Free, mixed, and hidden points}

Following \cite{stoyanov2017santalo}, in an open billiard, we say a state $(q,v) \in \hat{\Omega}$ is \emph{trapped} if it reflects infinitely many times in the forward direction. Otherwise, we say $(q,v)$ \emph{escapes}. We say $(q,v)$ is \emph{completely trapped} if it has infinitely many reflections in both directions (i.e. both $(q,v)$ and $(q, -v)$ are trapped). Denote the set of trapped states in $\hat{\Omega}$ by $\Trap(\hat{\Omega})$ \cite{stoyanov2017santalo}. For a point $q \in \hat{\Omega}$, denote by $\Tvectors(q) \subset S_q$ the set of vectors $v \in S_q$ such that $(q, v)$ is trapped. Denote by $\Fvectors(q) = S_q \backslash \Tvectors(q)$ the set of vectors $v \in S_q$ such that $(q,v)$ escapes. These sets are disjoint and satisfy 
$$0 \leq \omega_q(\Tvectors(q)), \omega_q(\Fvectors(q)) \leq \omega_q(S_q),$$
where $\omega_q$ is the Lebesgue measure on the unit sphere.
Let the \emph{free set} $\F \subset \Omega$ be the set of \emph{free points} $q \in Q$ that satisfy $\omega_q(T(q)) = 0$. Let the \emph{mixed set} $\M \subset \Omega$ be the set of \emph{mixed points} $q \in \Omega$ that satisfy $0 < \omega_q(\Tvectors(q)) < \omega_q(S(q))$. Let the \emph{hidden} set $\Hid$ be the set of \emph{hidden points} $q \in \Omega$ such that for $\omega_q$-almost all $v \in S(q)$, the state $(q,v)$ is trapped. The sets $\Hid, \M, \F$ are disjoint, and $\Hid \cup \M \cup \T = \Omega$.

\begin{remark}
	It is possible to have \enquote{almost hidden} points $q$ for which $\Tvectors(q) = \omega_q(S(q))$, but there is nevertheless at least one vector $v$ such that $(q, v)$ escapes. For example, if two circular obstacles are tangent to each other, then a trajectory passing between them may escape while every other trajectory from the same point is trapped. 
\end{remark}



\begin{prop}
	If $K$ is a billiard in $\Reals^2$ with a hidden set, the hidden set must be convex wherever it does not intersect $K$.
	
	\begin{proof}
		Suppose $\partial \Hid \cap \partial \M$ is strictly concave around some point $p \notin K$. Then by compactness of $K$, there exists an open neighborhood $A$ of $p$ such that $A \cap K = \emptyset$. Then consider a region $B$ containing $p$, bounded by $\partial \Hid \cap \partial \M$ and by a line segment with endpoints in $\partial \Hid \cap \partial \M \cap A$. For any $q \in \Int B, v \in \mathbb{S}^1$, there exists $t \in \Reals$ such that $\phi_t(q, v)$ is trapped, therefore $(q,v)$ is trapped. So $q$ is a hidden point, contradicting the assumption. 
	\end{proof}
\end{prop}

\subsection{Lemmas on convex sets and Lipschitz functions}

\begin{defn}
	A \emph{supporting line} is one that contains at least one point in $\partial H$, but does not separate any two points of $H$. A clockwise or anticlockwise \emph{supporting ray} to a convex set $\Hid$ is a ray beginning at a point $A \in \partial \Hid$, parallel to a supporting line through $A$, and in a direction such that $\Hid$ is always on the left or right respectively. 
\end{defn}

\begin{lem} \label{tangent lemma}
	For a given point $P$ outside a convex set $H$, there is exactly one clockwise supporting ray and one anticlockwise supporting ray to $H$ that passes through $P$.
\end{lem}
\begin{proof}
	Clearly there are at least two supporting lines (one on each side). A third supporting line would separate the two tangent points corresponding to the other two lines, which is a contradiction. It is easy to see that if one supporting ray has $\Hid$ on the left then the other has $\Hid$ on the right, so one is clockwise and the other is anti-clockwise. 
\end{proof}

\begin{lem} \cite{niculescu2006convex} \label{diff ae}
	Let $\Hid$ be a convex set with boundary arc-length \paramised by $\h(t)$. Then $\h(t)$ is continuous, semi-differentiable everywhere (i.e. the left and right derivatives $\partial_{\minus} \h$ and $\partial_{\plus} \h$ exist but may not be equal), and Lipschitz continuous everywhere. It is also differentiable everywhere except possibly at countably many points, and twice differentiable almost everywhere.
\end{lem}
\begin{proof}
	The proofs found in \cite{niculescu2006convex} are for convex functions. However they can be easily extended to convex sets (see e.g. \cite[Exercise 4, page 29]{niculescu2006convex}).
\end{proof}
Note that the boundary may be non-differentiable at a dense set of points \cite[Remark 1.6.2]{niculescu2006convex}. The second derivatives may not exist at uncountably many points (for example, if part of the boundary is the integral of the Cantor function \cite{darst1993hausdorff}). A convex set may contain dense sets of line segments and corners. 

\section{Construction} \label{S: construction}
\subsection{Parameterisation and tangential angle}
Let $\Hid$ be a convex set with perimeter $L$. Let $\h(t)$ be an anticlockwise, arc-length \paramation for $\partial H$, for $t \in [0, L]$. For $v \in \mathbb{S}^1$, define $\angle v \in [0, 2\pi)$ to be the anticlockwise angle from the $x$-axis to $v$. Without loss of generality, assume that $\h(t)$ is differentiable at $t = 0$ and that $\h'(0) = (1, 0)$. A \emph{tangential angle} or \emph{turning angle} of a curve at a point is the angle between the vector $\h'(0)$ and a supporting line through the point \cite{whewell1849intrinsic}. The tangential angle is a set-valued function of the parameter $t$, specifically the $\angle$ map applied to the \emph{subdifferential} of $\h$ \cite{niculescu2006convex}: 
$$\theta(t) = \angle \partial \h(t) = [\angle \partial_{\minus} \h(t), \angle \partial_{\plus} \h(t)].$$
Then $\theta(t)$ is monotonic if and only if the curve is convex \cite{abbena2006modern}.  The inverse relation is also a set valued function:
$$t(\theta) = \left\{t: \angle \partial_{\minus}\h(t) \leq \theta \leq \angle \partial_{\plus}\h(t) \right\}.$$
Define $\that(\theta), \tch(\theta)$ as the supremum and infimum of this set respectively. 
\begin{prop}
	The set valued function $t(\theta)$ is monotonic everywhere, in the sense that if $\theta_1 < \theta_2$ then $\that(\theta_1) < \tch(\theta_2)$. It is continuous and differentiable almost everywhere.
\end{prop}
\begin{proof}
The inverse (as a relation) of the subdifferential of a convex function $f$ is the subdifferential of the convex conjugate $f^*$ \cite[Theorem 1.7.3]{niculescu2006convex}. That is,
$$(\partial f)^{-1} = \partial f^*.$$
Since $f^*$ is a convex function, it has all the smoothness proporties of a convex function in \cref{diff ae}. This can easily be extended to convex curves. So $t(\theta)$ is monotonic everywhere, and single valued wherever $\h(t(\theta))$ is not a line segment. It is continuous and differentiable almost everywhere. 
\end{proof}

\subsection{Tangential coordinates}
Next we set up two different coordinate systems for $\overline{\Reals^2 \backslash \Hid}$. We can express any point $p$ in $\overline{\Reals^2 \backslash \Hid}$ using the clockwise or anticlockwise tangent rays to $\Hid$ through $p$. For $\theta \in [0, 2\pi)$ and $u \in \Reals$, define a function
$$\p(u,\theta) = \h(t) + (u - t) \begin{pmatrix}\cos \theta \\ \sin \theta \end{pmatrix}, \quad \text{ for any } t \in t(\theta).$$
This function is single valued and continuous, because if $t(\theta)$ is not single valued then $\h(t(\theta))$ is on a line segment in the direction of $(\cos \theta, \sin \theta)$. If $u > \widehat{t}(\theta)$ then the function represents the end of a rope of length $u$, with the other end tied at $\h(0)$, wrapped anticlockwise around $\Hid$ until its tangential angle is $\theta$.

\begin{lem}	
	The function $\p(u, \theta)$ is locally Lipschitz continuous with respect to $\theta$.	
	\begin{proof}
		For any $u \in \Reals$ and $\theta_1 < \theta_2 \in [0, 2\pi)$, let $t_i \in t(\theta_i)$, $\h_i = \h(t_i)$, $\p_i = \p(u, \theta_i)$, and $\phi_i =  \angle \p_1 \h_i \p_2$. By examining the three cases $u < t_1 < t_2$, $t_1 < u < t_2$ and $t_1 < t_2 < u$, it is easy to see that for either $i = 1$ or $i = 2$, we have $\phi_i < \theta_2 - \theta_1$. The triangle $\triangle \p_1 \h_i \p_2$ is contained in a larger isosceles triangle with apex $\h_i$, so we have
		$$\|\p_2 - \p_1\|	\leq \max\left\{ \|\p_2 - \h_i\| \phi_i, \|\p_1 - \h_i\| \phi_i \right\}   \leq K |\theta_2 - \theta_1|,$$ for some constant $K$.
	\end{proof}
\end{lem}
Note that $\p$ may be nondifferentiable at a dense set of values of $\theta$. To continue, we will need a fairly technical and recent generalization of derivatives and the implicit function theorem from Gowda \cite{gowda2004inverse,gowda2000algebraic}. 

\begin{defn}[$H$-differentiability and $H$-differentials] \cite{gowda2000algebraic}
	Let $f: X \rightarrow \Reals^n$ for an open set $X \subset \Reals^n$. We say that a non-empty set of matrices $T(x) \subset \Reals^{n \times n}$ is an $H$-differential of $f$ at $x$ if for every sequence $\{x_k\}$ converging to $x$, there exists a convergent subsequence $\{x_{k_j}\}$ and a matrix $M \in T(x)$ such that 
	$$\lim_{j \rightarrow \infty} \frac{f(x_{k_j}) - f(x) - M (x_{k_j} - x)}{\|x_{k_j} - x\|} = 0$$
	We say that $f$ is $H$-differentiable at $x$ if it has a $H$-differential at $x$. 	
\end{defn}

\begin{prop}
	Whenever $u \notin t(\theta)$, the function $\p$ is $H$-differentiable and the set
	\begin{align*}
	T_\p(u,\theta) &= \left\{\widecheck{M}, \widehat{M} \right\}\\
	&= \left\{ \begin{pmatrix}
	\cos \theta		&	-(u - \widecheck{t}(\theta)) \sin \theta\\
	\sin \theta		&	(u - \widecheck{t}(\theta)) \cos \theta
	\end{pmatrix}, \begin{pmatrix}
	\cos \theta		&	-(u - \widehat{t}(\theta)) \sin \theta\\
	\sin \theta		&	(u - \widehat{t}(\theta)) \cos \theta
	\end{pmatrix} \right\}
	\end{align*}
	is an $H$-differential of $\p$ at $(u,\theta)$. 
	
	\begin{proof}
		If $t(\theta)$ is single valued and differentiable, then $\p$ is differentiable at $(u, \theta)$ and its Jacobian matrix is
		$$J_\p(u, \theta) = \begin{pmatrix}
		\cos \theta		&	-(u - t(\theta)) \sin \theta\\
		\sin \theta		&	(u - t(\theta)) \cos \theta
		\end{pmatrix},$$
		so we are done. Suppose $\p$ is not differentiable at some $(u^*, \theta^*)$. Fix $\varepsilon > 0$, and let $\{(u_k, \theta_k)\}$ be a sequence of points converging to $(u^*, \theta^*)$. First we consider limits from the anticlockwise direction. Assume there is an infinite subsequence $k_j$ such that $\theta_{k_j} \leq \theta^*$. Then since $\p$ is differentiable for almost every $\theta$, it must be differentiable at some $\theta_{k_j}'$, where $\theta_{k_j} - \varepsilon < \theta_{k_j}' < \theta_{k_j}$ and $\theta_{k_j} \rightarrow \theta^*$. For convenience, we denote 
		\begin{align*}
		\p^*	&= \p(u^*, \theta^*),	& \p_j &= \p(u_{k_j}, \theta_{k_j}),
			& w_j  &= \left(u_{k_j} - u^*, \theta_{k_j} - \theta^* \right),	\\
		J_j	&= J_\p(u_{k_j}, \theta_{k_j}'),		& \p_j' &= \p(u_{k_j}, \theta_{k_j}'),
			&\quad w_j'  &= \left(u_{k_j} - u^*, \theta_{k_j}' - \theta^* \right).
		\end{align*}
		Using the triangle inequality, 
		\begin{align*}
		\left\|\p_j - \p^* - \widecheck{M} w_j \right\|
			&\leq \left\|\p_j' - \p^* - J_j w_j' + (\p_j - \p_j') + J_j(w_j' - w_j) + (J_j - \widecheck{M}) w_j \right\|\\
			&\leq \|\p'_j - \p^* - J_j w_j' \| + \|\p_j - \p_j'\| + \|J_j(w_j' - w_j)\| + \left\|(J_j - \widecheck{M}) w_j \right\|.
		\end{align*}
		Next we find upper bounds for each term. Note that $\ds \lim_{j \rightarrow \infty} J_j = \widecheck{M}$. So for sufficiently large $j$, we have
		$$\|\p_j - \p_j'\| < K \varepsilon, \quad \|w_j - w_j'\| < \varepsilon, \quad \|J_j - \widecheck{M}\| < \varepsilon,$$
		where $K$ is the Lipschitz constant for $\p$ with respect to $\theta$. So we have
		\begin{align*}
		\frac{\left\|\p_j - \p^* - \widecheck{M} w_j \right\|}{\|w_j\|} 
		&\leq \frac{\|\p'_j - \p^* - J_j w_j' \|  +  K \varepsilon + \|J_j\| \varepsilon + \left\|w_j \right\| \varepsilon}{\|w_j'\| - \varepsilon}.
		\end{align*}
		This holds for all $\varepsilon > 0$, so we have
		$$\left\| \lim_{j \rightarrow \infty} \frac{\p_j - \p^* - \widecheck{M} w_j}{\left\|w_j \right\|} \right\| \leq \lim_{j \rightarrow \infty} \frac{\|\p_j' - \p^* - J_j w_j' \|}{\left\|w_j' \right\|} = 0.$$
		We assumed above that there exist infinitely many $\theta_{k_j} \leq \theta^*$. If we assume instead that there are infinitely many $\theta_{k_j} \geq \theta^*$, we get
		$$\lim_{j \rightarrow \infty} \frac{\p(u_{k_j}, \theta_{k_j}) - \p(u^*, \theta^*) - \widehat{M} \left(
			\theta_{k_j} - \theta^*, u_{k_j} - u^* \right) }{\left\|\left(u_{k_j} - u^*, \theta_{k_j} - \theta^* \right) \right\|} = 0.$$
		Therefore $\left\{\widecheck{M}, \widehat{M} \right\}$ is an $H$-differential of $\p$ at $(u, \theta)$. 
	\end{proof}	
\end{prop}

Next we will use Gowda's inverse function theorem for $H$-differentiable functions. 
\begin{theorem}[Inverse function theorem for $H$-differentiable functions] \label{IVT} \cite{gowda2004inverse}
	Let $f: X \rightarrow \Reals^n$ be $H$-differentiable at every point $x \in X$ with an $H$-differential $T_f(x)$. Fix a point $x^* \in X$ and suppose 
	\begin{enumerate}
		\item If $f$ is differentiable at $x \in X$ then $f'(x) \in T(x)$.
		\item The set $T(x)$ is compact.
		\item The map $x \mapsto T(x)$ is upper hemicontinuous.
		\item $T(x^*)$ consists of matrices with only positive or only negative determinants. 
		\item The topological index of $f$ at $x^*$ is the same as the sign of the determinants of matrices in $T(x^*)$.
	\end{enumerate}	
	Then there is a continuous, locally Lipschitz inverse function $f^{-1}$ on a neighborhood of $y^* = f(x^*)$, with the following $H$-differential:
	$$T_{f^{-1}}(y^*) = \{M^{-1}: M \in T_f(x^*)\}.$$
\end{theorem}
	\noindent Note that when $u \in t(\theta)$, we have $\p(u, \theta) = \h(u) \in \partial \Hid$. If we define two sets
	\begin{align*}
	X_{\plus} &= \{(u, \theta): 0 \leq \theta < 2\pi, u > \widehat{t}(\theta)\},\\
	X_{\minus}	&= \{(u, \theta): 0 \leq \theta < 2\pi, u < \widecheck{t}(\theta)\},
	\end{align*}
	then $\p: X_{\plus} \rightarrow \Reals^2 \backslash \Hid$ and $\p: X_{\minus} \rightarrow \Reals^2 \backslash \Hid$ are both bijections (this follows from \cref{tangent lemma}). 
\begin{prop}
	For all $(x,y) \in \Reals^2 \backslash \Hid$, there exist continuous, locally Lipschitz inverse functions $(u_\pm, \theta_\pm) \in X_\pm$. These functions have an $H$-differential:
	\begin{align*}
	T_{(u_\pm, \theta_\pm)}(x,y) &= \left\{\left(\widecheck{M} \right)^{-1}, \left(\widehat{M}\right)^{-1} \right\}\\
	&= \left\{ \begin{pmatrix}
	\cos \theta_\pm		&	\sin \theta_\pm \vspace{3pt}\\ 
	\ds \frac{-\sin \theta_\pm}{u - \tch(\theta_\pm)}		&	\ds \frac{\cos \theta_\pm}{u - \tch(\theta_\pm)}
	\end{pmatrix}, \begin{pmatrix}
	\cos \theta_\pm		&	\sin \theta_\pm \vspace{3pt}\\ 
	\ds \frac{-\sin \theta_\pm}{u - \that(\theta_\pm)}		&	\ds \frac{\cos \theta_\pm}{u - \that(\theta_\pm)}
	\end{pmatrix} \right\}.
	\end{align*}
	
	\begin{proof}
		First we check that the conditions of \cref{IVT} are satisfied. Fix a point $(u, \theta) \in X_\pm$. 
			\begin{enumerate}
				\item We already showed that if $\p$ is differentiable then $J_{\p}(u, \theta) \in T_{\p}(u, \theta)$.
				\item Clearly the $H$-differential is compact, since it has only one or two elements. 
				\item The map $(u, \theta) \mapsto T_{\p}(u, \theta)$ is upper hemicontinuous, because if $(u_k, \theta_k) \rightarrow (u, \theta)$ then for any sequence of matrices $M_k \in T(u_k, \theta_k)$, if $M_k \rightarrow M$ then $M \in T(u, \theta)$.
				\item Each matrix in $T_{\p}(u, \theta)$ has determinant $u - \widehat{t}(\theta)$ or $u - \widecheck{t}(\theta)$. These are always positive for $(u, \theta) \in X_{\plus}$ and always negative for $(u, \theta) \in X_{\minus}$.
				\item We use the properties of topological degree from \cite{gowda2004inverse}. If $\p$ is differentiable at $(u, \theta) \in X_\pm$ then the topological index is $\deg \left(\p, X_\pm, (\theta, u) \right) = \pm 1$. Otherwise, it is still $\pm 1$ by the nearness property.
			\end{enumerate}
		So the conditions of \cref{IVT} are satisfied and the result follows.
	\end{proof}
\end{prop}
\noindent So we have $\nabla u_\pm(x,y) = (\cos \theta_\pm, \sin \theta_\pm)$ for all $x, y \in \Reals^2 \backslash \Hid$. Furthermore, whenever $t(\theta_\pm)$ is single valued, we have $\ds \nabla \theta_\pm(x,y) = \frac{(- \sin \theta_\pm, \cos \theta_\pm)}{u_\pm - t(\theta)}$.

\subsection{Potential function}

We construct a potential function $\varphi(x,y)$ on $\Reals^2 \backslash \Hid$, the level curves of which will form the boundary of the required billiard. The value of $\varphi(x,y)$ represents the length of rope needed to wrap around $\Hid$ and the point $(x,y)$. The supporting lines of $\partial \Hid$ through $(x,y)$ will intersect $\partial \Hid$ at points $T_{\minus}$ and $T_{\plus}$ (if the line intersects at an interval, choose an arbitrary point from it to be $T_{\minus}$). Then $\varphi(x,y)$ is the sum of the distances from $(x,y)$ to each tangent point $T_{\plus}$, $T_{\minus}$, plus the arc length of $\partial \Hid$ between $T_{\plus}, T_{\minus}$ on the opposite side of $(x,y)$.

\begin{prop}
	The function $\varphi(x,y)$ is continuously differentiable, and its gradient bisects the angle between the two supporting lines through $(x,y)$.
\end{prop}

\begin{proof}

For the case $y \geq 0$, we split the rope into two curves: one of length $\varphi_{\minus}(x,y) = L - u_{\minus}(x,y)$ running clockwise from $\h(0)$ through $T_{\minus}$ to $(x,y)$, and the other of length $\varphi_{\plus}(x,y) = u_{\plus}(x,y)$ running anti-clockwise from $\h(0)$ through $T_{\plus}$ to $(x,y)$. Choose arbitrary points $t_{\minus} \in t(\theta_{\minus})$ and $t_{\plus} \in t(\theta_{\plus})$. The potential function is 
\begin{align*}
\varphi(x,y)	&= \left\|(x,y) - T_{\plus} \right\| + \|(x,y) - T_{\minus}\| + \int\limits_{\partial \Hid[T_{\minus}, T_{\plus}]} dt\\
&= \left\|p(u_{\plus}, \theta_{\plus}) - \h(t(\theta_{\plus})) \right\| + \|p(u_{\minus}, \theta_{\minus}) - \h(t(\theta_{\minus}))\| + \left(L - |t_{\minus} - t_{\plus}| \right)\\
&= \left|\varphi_{\plus} - t_{\plus} \right| + \left|L - \varphi_{\minus} - t_{\minus} \right| + \left(L - (t_{\minus} - t_{\plus}) \right)\\
&= \varphi_{\plus} + \varphi_{\minus}.
\end{align*}
For the case $y < 0$, we split the rope at the point $\h(t(\pi))$, which has the largest $y$ component on $\Hid$. So the two parts have lengths $\varphi_{\plus} = u_{\plus}(x,y) - t(\pi)$ and $\varphi_{\minus} = t(\pi) - u_{\minus}(x,y)$.
Choose arbitrary points $t_{\minus} \in t(\theta_{\minus}(x,y))$ and $t_{\plus} \in t(\theta_{\plus}(x,y))$. The potential function is 
\begin{align*}
\varphi(x,y)	&= \|(x,y) - T_{\plus}\| + \|(x,y) - T_{\minus}\| + \int\limits_{\partial \Hid[T_{\minus}, T_{\plus}]} dt\\
				&= \left\|p(u_{\plus}, \theta_{\plus}) - \h(t(\theta_{\plus})) \right\| + \|p(u_{\minus}, \theta_{\minus}) - \h(t(\theta_{\minus}))\| + \left(t_{\plus} - t_{\minus} \right)\\
				&= \left|\varphi_{\plus} + t(\pi) - t_{\plus} \right| + \left|t(\pi) - \varphi_{\minus} - t_{\minus} \right| + \left(t_{\plus} - t_{\minus} \right)\\
				&= \varphi_{\plus} + \varphi_{\minus}.
\end{align*}
So $\varphi(x,y) = \varphi_{\plus}(x,y) + \varphi_{\minus}(x,y)$ for all $(x,y) \in \Reals^2 \backslash \Hid$. Although $\varphi_{\minus}$ and $\varphi_{\plus}$ are piecewise defined and not continuous at $y = 0$, their sum is clearly continuous. It is also continuously differentiable everywhere, with gradient
$$\nabla \varphi = \begin{pmatrix}
\cos \theta_{\plus} \\ \sin \theta_{\plus}
\end{pmatrix} - \begin{pmatrix}
\cos \theta_{\minus} \\ \sin \theta_{\minus}
\end{pmatrix}.$$
This clearly bisects the angle between the two supporting lines, which have directions $(\cos \theta_{\plus}, \sin \theta_{\plus})$ and $(- \cos \theta_{\minus}, -\sin \theta_{\minus})$. 
\end{proof}


\section{Proof of main theorem}

Let $\Hid$ be a convex set with perimeter $L$, let $\lambda > 0$, and let $\partial K_\lambda$ be the level curve $\varphi(x,y) = 2 \lambda + L$. We prove the main theorem in three separate propositions. 

\begin{prop}
	The level curve $K_\lambda$ satisfies 
	$$\|(x,y) - h\| \geq \lambda, \text{ for all } (x,y) \in \partial K_\lambda, h \in \Hid.$$
	\begin{proof}
		For a point $x,y$, choose $t_\pm \in t(\theta_\pm)$. Then for any $t^* \in (t_{\minus}, t_{\plus})$, using convexity of the hidden set and the triangle inequality, we have 
		\begin{align*}
			\varphi(x,y) &= \|(x,y) - T_{\plus}\| + \|(x,y) - T_{\minus}\| + \int\limits_{\partial \Hid[T_{\minus}, T_{\plus}]} ds\\
			& \leq \|(x,y) - \h(t_{\minus})\| + \|(x,y) - \h(t_{\plus})\|\\
			&+ L - \|\h(t_{\minus}) - \h(t^*)\| - \|\h(t_{\plus}) - \h(t^*)\|\\
			& \leq 2\|(x,y) - \h(t^*)\| + L.
		\end{align*}
		In particular, on the level curve $\varphi = 2 \lambda + L$, we have $\ds 2\lambda + L =  \varphi(x,y) \leq 2 \min_{h \in \Hid} \|(x, y) - h\| + L$ and the result follows. 
	\end{proof}
\end{prop}

\begin{prop}
	Each level curve is strictly convex.	
	\begin{proof}
		Let $\gamma(\tau)$ \paramise the boundary $\partial K_\lambda$ anticlockwise. The tangential angle is $\angle \gamma'(\tau) = \frac{1}{2}(\theta_{\plus}(\gamma(\tau)) + \theta_{\minus}(\gamma(\tau)))$. Each angle $\theta_{\minus}, \theta_{\plus}$ is nondecreasing in $\tau$, and at least one of them is increasing (otherwise the two tangent lines would be parallel). So $\angle \gamma'(\tau)$ is strictly increasing, therefore the curve $\varphi(x,y) = c$ is strictly convex. In fact we can calculate the curvature directly wherever it exists. The curvature of a level curve $\varphi(x,y) = c$ is 
		$$\kappa_\varphi(x,y) = \frac{\varphi_x^2 \varphi_{yy} - 2 \varphi_x \varphi_y \varphi_{xy} + \varphi_y^2 \varphi_{xx}}{(\varphi_x^2 + \varphi_y^2)^{3/2}},$$
		when the second derivatives exist. A simple but very long calculation shows that the curvature of $K_\lambda$ is equal to 
		$$\kappa_\varphi(x,y) = \frac{|u_{\plus} - t_{\plus}| + |u_{\minus} - t_{\minus}|}{2 |u_{\plus} - t_{\plus}| |u_{\minus} - t_{\minus}|} \left| \sin \left( \frac{\theta_{\plus} - \theta_{\minus}}{2}\right) \right|,$$
		whenever $\widehat{t}(\theta_\pm) = \widecheck{t}(\theta_\pm)$.
		This is always positive, because $\theta_{\plus}$ and $\theta_{\minus}$ cannot be equal unless $(x,y) \in \partial \Hid$. The curvature tends to zero as $\|(x,y)\|$ approaches infinity, and it approaches the curvature of $\partial \Hid$ at $h$ (if it exists) as $(x,y) \rightarrow h$. 
	\end{proof}
\end{prop}

\begin{prop}	
	Let $\Hid$ be a convex set and let $\K_\lambda$ be a billiard table with boundary $\varphi(x,y) = L + 2\lambda$. Then the phase space $\hat{\Omega}$ of the billiard flow inside $\K_\lambda$ is split into two disjoint subsets $\hat{\Omega} = \hat{\Omega}_1 \cup \hat{\Omega}_2$. Every trajectory in $\hat{\Omega}_1$ intersects $\Hid$ after every reflection, while every trajectory in $\hat{\Omega}_2$ never intersects $\Hid$. 
\end{prop}

\begin{proof}
	Consider a billiard trajectory tangent to $\Hid$ at $\h(t_{\plus})$ and colliding with $\partial K$ at $(x,y)$. The normal vector to $\partial K$ at $(x,y)$ is $\nabla \varphi$, which bisects the vectors $(\cos \theta_{\minus}, \sin \theta_{\minus})$ and $(\cos \theta_{\plus}, \sin \theta_{\plus})$ at $(x,y)$. The angle of incidence is $\frac{\theta_{\minus} - \theta_{\plus}}{2}$. So the reflected trajectory must be tangent to $\Hid$ at $\h(t(\theta_{\minus}))$. Next consider a trajectory coming from inside $\Hid$ and colliding with $\partial K$ at $(x,y)$. This must have a smaller angle of incidence and reflection, so it will return to $\Hid$ after one reflection. Similarly a trajectory that does not intersect $\Hid$ before reflecting at $(x,y)$ will have a greater angle of incidence and reflection, so it will not intersect $\Hid$ after reflecting. Thus the phase space inside $\K_\lambda$ is split as required.
\end{proof}
This completes the proof of the main theorem.

%
%

\begin{proof}[Proof of \cref{cor1}]
	Let $\Hid$ be a potential hidden set. Let $\focusA, \focusB$ be any two points on the boundary with tangent lines $L_A, L_B$. Let $K_1 \subset \Hid$ share one side of the boundary with $\Hid$ between $\focusA$ and $\focusB$ (otherwise this component is arbitrary). If the intersection $L_A \cap L_B = P$ is a point on the opposite side of $K_1$ from $\Hid$, then let 
	$$\lambda_{\max} = \min\{\|A - P\|, \|B - P\|\}.$$
	If the intersection $L_A \cap L_B$ is a point on the same side of $K_1$ as $\Hid$, or if $L_A, L_B$ are parallel, then let $\lambda_{\max} = \infty$. Let $L$ be the arc length of $\partial \Hid$ from $A$ to $B$ (the part not overlapping $K_1$). Let $R_A, R_B$ be the half-planes on the opposite side of $\Hid$ from $L_A, L_B$ respectively. Let $R_H$ be the region bounded by $L_A, L_B$ and $\partial \Hid$, and let $R_0$ be the region bounded by $L_A, L_B$ and $\partial K_1$. Then for $(x,y) \in \Reals^2 \backslash (\Hid \cup R_0)$, let $T_{\minus} = A$ if $(x,y) \in R_A$ and otherwise choose $T_{\minus} \in \partial \Hid$ so that an anticlockwise supporting ray intersects $(x,y)$. Similarly, let $T_{\plus} = B$ if $(x,y) \in R_B$ and otherwise choose $T_{\plus} \in \partial \Hid$ so that an anticlockwise supporting ray from $T_{\plus}$ intersects $(x,y)$. Then define the potential function on $\Reals^2 \backslash \Hid$ by
	\begin{align*}
	\varphi(x,y)	&= \|(x,y) - T_{\minus}\| + \|(x,y) - T_{\plus}\| + \int\limits_{\partial \Hid[T_{\minus}, T_{\plus}]} ds.
	\end{align*}
	This is very similar to the original potential function, except for the altered tangent points. By modifying the construction in section \ref{S: construction} it is easy to see that the level curve $\varphi(x,y) = 2 \lambda + L$ splits the phase space around $\Hid$, everywhere except the region $R_0$. For any $\lambda \in (0,\lambda_{\max})$, the level curve intersects $L_A$ and $L_B$ orthogonally at $A', B'$ respectively. Extend the curve back around as in \cref{fig2}({\textsc{b}}) to form the boundary of $\K_\lambda$. Clearly any trajectory passing through $L_A$ between $A$ and $A'$ (or passing through $L_B$ between $B$ and $B'$) will never reach the hidden set. 
	
		
\end{proof}

\section{Remarks and future research}

\begin{remark}
	In \cref{main theorem}, if the hidden set $\Hid$ is a polygon with finitely many sides, then $\partial K_\lambda$ will be entirely composed of elliptical arcs.
\end{remark}

\begin{remark}
	For \cref{main theorem} and both corollaries, in the limit as $\lambda \rightarrow 0$, the billiard $\partial K_\lambda$ approaches $\partial \Hid$ itself. 
\end{remark}
\noindent The constructions presented in this paper are not unique, and some of the restrictions given can be relaxed.

\begin{figure}
	\centering
	\subcaptionbox{\label{fig:Triangle} The lower part of this billiard has been shifted inwards. The elliptical arcs are labelled by their focii, e.g. the arc $AB$ is part of an ellipse with focii $A, B$. In the limit as $A' \rightarrow A$ and $B' \rightarrow B$, the billiard approaches a Bunimovich mushroom.}
	{\includegraphics[width = \textwidth, clip = true, trim = 5.5cm 14.5cm 0cm 4cm]{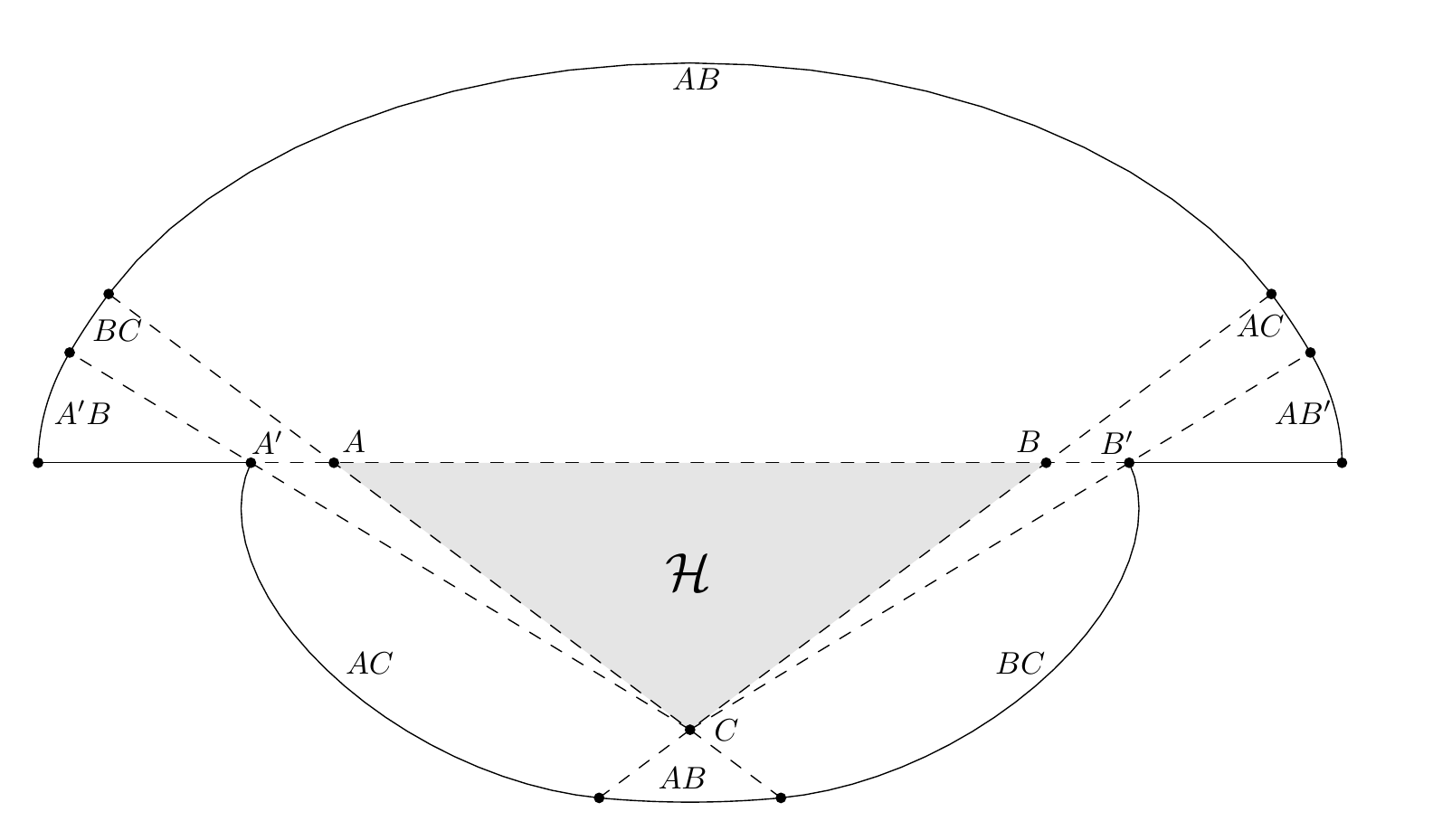}}
	\subcaptionbox{\label{fig:ParabolicLivshits} A parabolic billiard with a hidden set. The dotted lines are trajectories from the focus of the two parabolas. The free and trapped sets are shown at points $p, q$. }
	{\includegraphics[width = 0.45\textwidth]{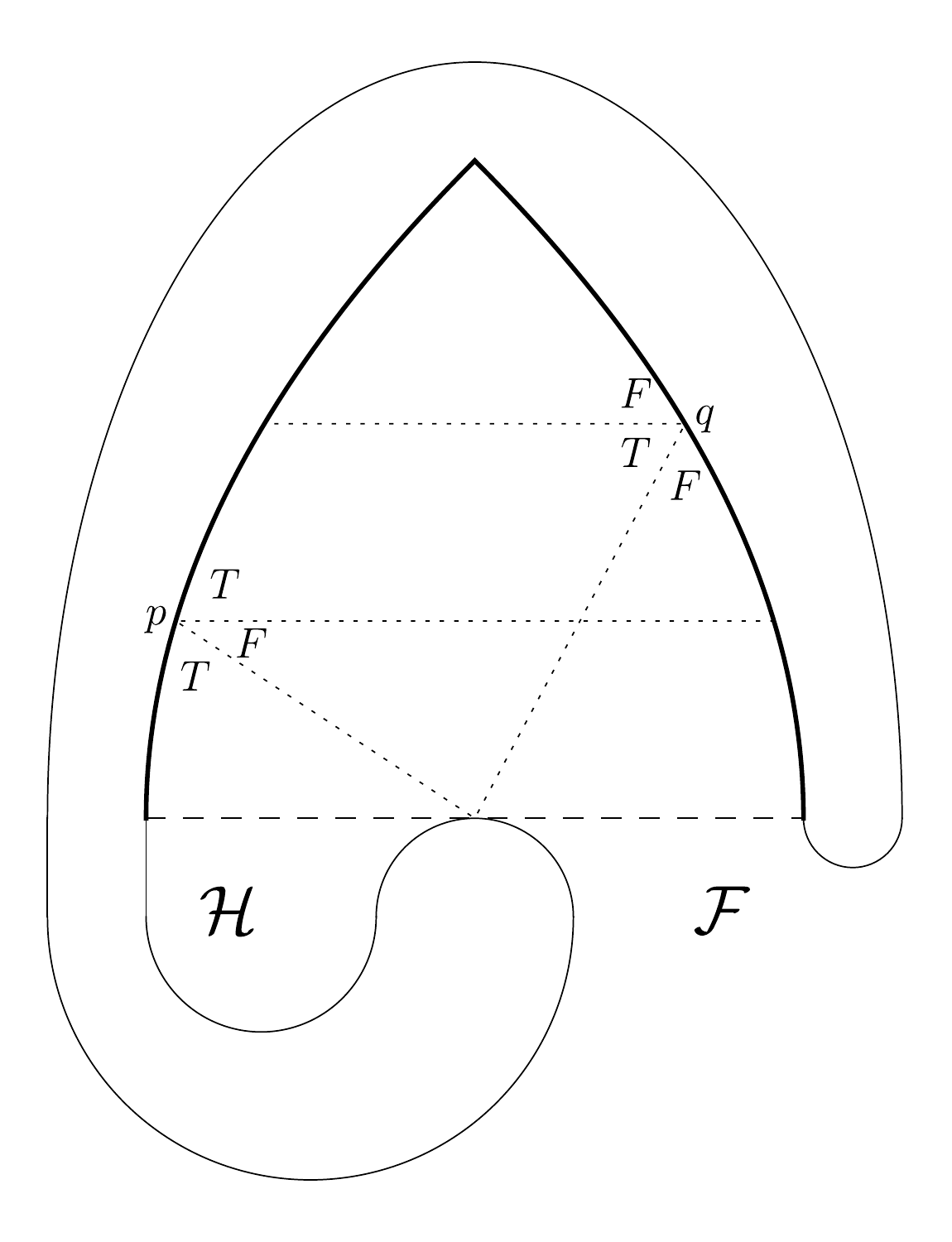}}
	\hspace{0.5cm}
	\subcaptionbox{\label{fig:ConcaveHidden} The concave part of the hidden set is covered by a thin billiard obstacle. Since the boundary is a polygon, the billiard $\partial \K_\lambda$ is a piecewise union of ellipses.}
	{\includegraphics[width = 0.4\textwidth, clip = true, trim = 6cm 16cm 6cm 4cm]{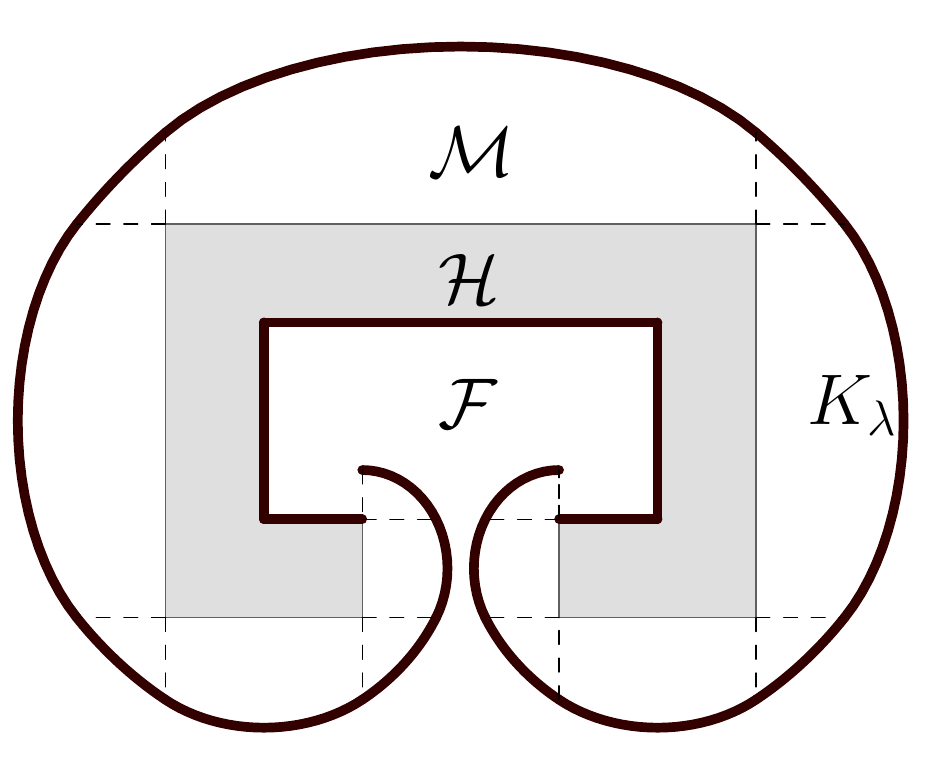}}
	\caption{\label{fig:MoreExamples}}
\end{figure}

\subsection{Shifting pieces of $K_\lambda$.}

By shifting parts of $\partial K_\lambda$ in and out, and filling in the resulting gaps with other curves, it is possible to create piecewise differentiable billiards with the same hidden set. The general method here will likely be very complicated, so we will only provide one example, \cref{fig:MoreExamples}(\textsc{a}), rather than going into detail. The original Bunimovich mushrooms can be constructed in this way: part of the billiard is shifted inwards until it touches the hidden set. 

\subsection{Constructions with two or more reflections} \label{Other constructions}

We have assumed that a trajectory leaving $\Hid$ will reflect exactly once and then return to $\Hid$. But there may be billiard systems where trajectories can reflect two or more times outside $\Hid$ before returning to $\Hid$. \cref{fig:MoreExamples}(\textsc{b}) shows one example using two parabolic curves with the same focus and directrix. Trajectories leaving $\Hid$ reflect at least twice before returning to $\Hid$, and they can never reach $\F$. There may be much more complicated examples with two or more reflections.

\subsection{Concave hidden sets}
	The above constructions can be extended to concave hidden sets, provided that certain concave parts of the boundary are covered by a billiard obstacle. \cref{fig:MoreExamples}(\textsc{c}) shows an example. We conjecture that this is possible for any set, although it may be difficult to say exactly which parts of the boundary must be covered and find bounds on $\lambda$ so that the billiard does not intersect itself. 





\bibliographystyle{amsplain}
\bibliography{Billiards}

\end{document}